\documentclass{amsart}
\usepackage{amsmath}

\newtheorem{theorem}{Theorem}

\newtheorem{corollary}[theorem]{Corollary}

\newtheorem{definition}[theorem]{Definition}
\newtheorem{example}[theorem]{Example}

\newtheorem{proposition}[theorem]{Proposition}
\newtheorem{remark}[theorem]{Remark}

\numberwithin{equation}{section}
\def\({\left ( }
\def\){\right )}
\def\<{\left < }
\def\>{\right >}

\input{tcilatex}

\begin{document}
\title{\textbf{Topologies and approximation operators induced by binary
relations}}
\author{Nurettin BA\u{G}IRMAZ, A. Fatih \"{O}ZCAN Hatice TA\c{S}BOZAN and 
\.{I}lhan \.{I}\c{C}EN}
\address{MARDIN ARTUKLU UNIVERSITY, MARDIN, TURKEY }
\email{nurettinbagirmaz@artuklu.edu.tr, tel:+90 482 213 40 02, Fax: +90 482
213 40 04 }
\thanks{}

\begin{abstract}
Rough set theory is an important mathematical tool for dealing with
uncertain or vague information. This paper studies some new topologies
induced by a binary relation on universe with respect to neighborhood
operators. Moreover, the relations among them are studied. In additionally,
lower and upper approximations of rough sets using the binary relation with
respect to neighborhood operators are studied and examples are given.
\end{abstract}

\maketitle

\vspace{2cm}

\section{Introduction}

Rough set theory was introduced by Pawlak as a mathematical tool to process
information with uncertainty and vagueness \cite{pawlak1982}.The rough set
theory deals with the approximation of sets for classification of objects
through equivalence relations. Important applications of the rough set
theory have been applied in many fields, for example in medical science,
data analysis, knowledge discovery in database \cite{pawlak2002,
pawlak2007,thivager}.

The original rough set theory is based on equivalence relations, but for
practical use, needs to some extensions on original rough set concept. This
is to replace the equivalent relation by a general binary relation \cite%
{yao,yao1,lashin,jarvinen,xu}. Topology is one of the most important
subjects in mathematics. Many authors studied relationship between rough
sets and topologies based on binary relations \cite{pei,allam,li,qin}. In
this paper, we proposed and studied connections between topologies generated
using successor, predecessor, successor-and-predecessor, and
successor-or-predecessor neighborhood operators as a subbase by various
binary relations on a universe, respectively. In addition to this, we
investigate connection between lower and upper approximation operators using
successor, predecessor, successor-and-predecessor, and
successor-or-predecessor neighborhood operators by various binary relations
on a universe, respectively. Moreover, we give several examples for a better
understanding of the subject.

\section{Preliminaries}

In this section, we shall briefly review basic concepts and relational
propositions of the relation based rough sets and topology. For more
details, we refer to \cite{yao,yao1,pei}.

\subsection{Basic properties relation based rough approximations and
neighborhood operators}

\ 

In this paper, we always assume that $U$ is a finite universe, i.e., a
non-empty finite set of objects, $R$ is a binary relation on $U$, i.e., a
subset of $U^{2}$ $=U\times U$ \cite{pei}.

$R$ is serial if for each $x\in U$, there exists $y\in U$ such that $%
(x,y)\in R$; $R$ is inverse serial if for each $x\in U$, there exists $y\in
U $ such that $(y,x)\in R$; $R$ is reflexive if for each $x\in U$, $(x,x)\in
R$; $R$ is symmetric if for all $x,y\in U$, $(x,y)\in R$ implies $(y,x)\in R$
$;R$ is transitive if for all $x,y,z\in U$, $(x,y)\in R$ and $(y,z)\in R$
imply $(x,z)\in R$ [9].

$R$ is called a pre-order (relation) if $R$ is both reflexive and
transitive; $R$ is called a similarity (or, tolerance) relation if $R$ is
both reflexive and symmetric; $R$ is called an equivalence relation if $R$
is reflexive, symmetric and transitive \cite{pei}.

Given a universe $U$ and a binary relation $R$ on $U$, $x,y\in U,$ the sets\ 
\begin{align*}
R_{s}(x) & =\left\{ y\in U|(x,y)\in R\right\} , \\
R_{p}(x) & =\left\{ y\in U|(y,x)\in R\right\} , \\
R_{s\wedge p}(x) & =\left\{ y\in U|(x,y)\in R\text{ }and\text{ }(y,x)\in
R\right\} =R_{s}(x)\cap R_{p}(x), \\
R_{s\vee p}(x) & =\left\{ y\in U|(x,y)\in R\text{ }or\text{ }(y,x)\in
R\right\} =R_{s}(x)\cup R_{p}(x)
\end{align*}

\hspace{-0.45cm}are called the successor, predecessor,
successor-and-predecessor, and successor-or-predecessor neighborhood of $x$,
respectively, and the following four set-valued operators from $U$ to the
power set $P\left( U\right) $

\begin{align*}
R_{s} & :x\mapsto R_{s}(x), \\
R_{p} & :x\mapsto R_{p}(x), \\
R_{s\wedge p} & :x\mapsto R_{s\wedge p}(x), \\
R_{s\vee p} & :x\mapsto R_{s\vee p}(x)
\end{align*}

\hspace{-0.45cm}are called the successor, predecessor,
successor-and-predecessor, and successor-or-predecessor neighborhood
operators, respectively. Relationships between these neighborhood systems
can be expressed as:

\begin{align*}
R_{s\wedge p}(x)& \subseteq R_{s}(x)\subseteq R_{s\vee p}(x), \\
R_{s\wedge p}(x)& \subseteq R_{p}(x)\subseteq R_{s\vee p}(x)
\end{align*}%
\cite{yao1,pei}.

\begin{definition}
\cite{yao1} Let $R$ be a binary relation on $U$. The ordered pair $(U,R)$ is
called a (generalized) approximation space based on the relation $R$. For $%
X\subseteq U$, the lower and upper approximations of $X$ \ with respect to $%
R_{s}(x),$ $R_{p}(x),$ $R_{s\wedge p}(x),$ $R_{s\vee p}(x)$ are respectively
defined as follows:%
\begin{align*}
\underline{apr}_{R_{s}}(X)& =\left\{ x\in U|R_{s}(x)\subseteq X\right\} , \\
\underline{apr}_{R_{p}}(X)& =\left\{ x\in U|R_{p}(x)\subseteq X\right\} , \\
\underline{apr}_{R_{s\wedge p}}(X)& =\left\{ x\in U|R_{s\wedge
p}(x)\subseteq X\right\} , \\
\underline{apr}_{R_{s\vee p}}(X)& =\left\{ x\in U|R_{s\vee p}(x)\subseteq
X\right\} , \\
\overline{apr}_{R_{s}}(X)& =\left\{ x\in U|R_{s}(x)\cap X\neq \emptyset
\right\} , \\
\overline{apr}_{R_{p}}(X)& =\left\{ x\in U|R_{p}(x)\cap X\neq \emptyset
\right\} , \\
\overline{apr}_{R_{s\wedge p}}(X)& =\left\{ x\in U|R_{s\wedge p}(x)\cap
X\neq \emptyset \right\} , \\
\overline{apr}_{R_{s\vee p}}(X)& =\left\{ x\in U|R_{s\vee p}(x)\cap X\neq
\emptyset \right\} .
\end{align*}
\end{definition}

In Pawlak's classical rough set theory for lower and upper approximations
operators, an equivalence relation $R$ is used. In this case, four
neighborhood operators become the same, i.e., $R_{s}(x)=$ $R_{p}(x)=$ $%
R_{s\wedge p}(x)=$ $R_{s\vee p}(x)=[x]_{R}$, where $[x]_{R}$ is the
equivalence class containing $x$.

\begin{proposition}
\cite{yao} For an arbitrary neighborhood operator in an approximation space $%
(U,R),$ the pair of approximation operators satisfy the following properties:
\end{proposition}

\ 

$(L0)$ $\underline{apr}(X)=\left( \overline{apr}(X^{c})\right) ^{c},$ \ 

$(U0)$ $\overline{apr}(X)=\left( \underline{apr}(X^{c})\right) ^{c},$ \ 

$(L1)$ $\underline{apr}(U)=U,$ \ 

$(U1)$ $\underline{apr}(\emptyset)=\emptyset,$ \ 

$(L2)$ $\underline{apr}(X\cap Y)=\underline{apr}(X)\cap\underline {apr}(Y),$
\ 

$(U2)$ $\overline{apr}(X\cup Y)=\overline{apr}(X)\cup\overline{apr}(Y).$

\ 

\hspace{-0.45cm}where $X^{c}$ is the complement of $X$ with respect to $U.$
\ 

Moreover, if $R$ is reflexive, then

\ 

$(L3)$ $\underline{apr}(X)\subseteq$ $X$, \ 

$(U3)$ $X\subseteq\overline{apr}(X).$

\ 

\hspace{-0.45cm}If $R$ is symmetric, then

\ 

$(L4)$ $X\subseteq\underline{apr}(\overline{apr}(X)),$ \ 

$(U4)$ $\overline{apr}(\underline{apr}(X))\subseteq$ $X$.

\ 

\hspace{-0.45cm}If $R$ is transitive, then

\ 

$(L5)$ $\underline{apr}(X)\subseteq\underline{apr}(\underline{apr}(X)),$ \ 

$(U5)$ $\overline{apr}(\overline{apr}(X))\subseteq$ $\overline {apr}\left(
X\right) $.

\subsection{The concept of a topological space}

\ 

In this section, we give some basic information about the topology \cite%
{janich,pei}.

\begin{definition}
\cite{janich} A topological space is a pair $(U,T)$ consisting of a set $U$
and a set $T$ of subsets of $U$ (called "open sets"), such that the
following axioms hold: \ 

$(A1)$ Any union of open sets is open.

$(A2)$ The intersection of any two open sets is open.

$(A3)$ $\emptyset$ and $U$ are open.
\end{definition}

The pair $(U,T)$ speaks simply of a topological space $U.$

\begin{definition}
\cite{pei} Let $U$ be a topological space.

\begin{enumerate}
\item $X\subseteq U$ is called closed when $X^{c}$ is open.

\item $X\subseteq U$ is called a neighborhood of $x\in X$ if there is an
open set $V$ with $x\in V\subseteq X.$

\item A point $x$ of a set $X$ is an interior point of $X$ if $X$ is a
neighborhood of $x$, and the set of all interior points of $X$ is called the
interior of $X$. The interior of $X$ is denoted by $\overset{o}{X}$.

\item The closure of a subset $X$ of a topological space $U$ is the
intersection of the family of all closed sets containing $X$. The closure of 
$X$ is denoted by $\overset{-}{X}$.
\end{enumerate}
\end{definition}

In topological space $U$, the operator 
\begin{equation*}
\overset{o}{}:P\left( U\right) \rightarrow P\left( U\right) ,\text{ }X\mapsto%
\overset{o}{X}
\end{equation*}
is an interior operator on $U$ and for all $X,Y\subseteq U$ the following
properties hold:

\ 

I1) \bigskip\ $\overset{o}{U}=U,$ \ 

I2) $\overset{o}{X}\subseteq X,$ \ 

I3)$\overset{o}{\left( \overset{o}{X}\right) }=X,$ \ 

I4)$\left( \overset{o}{X\cap Y}\right) =\left( \overset{o}{X}\right)
\cap\left( \overset{o}{Y}\right) .$

In topological space $U$, the operator 
\begin{equation*}
\overset{-}{}:P\left( U\right) \rightarrow P\left( U\right) ,\text{ }X\mapsto%
\overset{-}{X}
\end{equation*}
is a closure operator on $U$ and for all $X,Y\subseteq U$ the following
properties hold: \ 

C1) $\overset{-}{\emptyset}=\emptyset,$ \ 

C2) $X\subseteq\overline{X}$ \ 

C3) $\overline{\overline{X}}=X,$ \ 

C4) $\overline{X\cup Y}=\overline{X}\cup\overline{Y}.$

In a topological space ($U,T)$ a family $\mathcal{B}\subseteq T$ of sets is
called a base for the topology $T$ if for each point $x$ of the space, and
each neighborhood $X$ of $x$, there is a member $V$ of $\mathcal{B}$ such
that $x\in V\subseteq X$. We know that a subfamily $\mathcal{B}$ of a
topology $T$ is a base for $T$ if and only if each member of $T$ is the
union of members of $\mathcal{B}$. Moreover, $\mathcal{B}\subseteq P(U)$
forms a base for some topologies on $U$ if and only if $\mathcal{B}$
satisfies the following conditions: \ 

B1) $U=\bigcup\left\{ B|B\in\mathcal{B}\right\} $, \ 

B2) For every two members $X$ and $Y$ of $\mathcal{B}$ and each point $x\in
X\cap Y$, there is $Z\in\mathcal{B}$ such that $x\in Z\subseteq X\cap Y$. \ 

Also, a family $\mathcal{S}\subseteq T$ of sets is a subbase for the
topology $T$ if the family of all finite intersections of members of $%
\mathcal{S}$ is a base for $T$. Moreover, $\mathcal{S}\subseteq P(U)$ is a
subbase for some topology on $U$ if and only if $\mathcal{S}$ satisfies the
following condition: \ 

S0) $U=\bigcup\left\{ S|S\in\mathcal{S}\right\} $

\section{Correspondence between generating topologies by relations}

In this section, we investigate connections between topologies generated
using successor, predecessor, successor-and-predecessor, and
successor-or-predecessor neighborhood operators as a subbase by various
binary relations on a universe, respectively.

Let $R$ be a binary relation on a given universe $U$. Sets

\begin{center}
$\mathcal{S}_{i}=\bigcup \left\{ R_{i}(x)|x\in U\right\}$ ,where $%
i:s,p,s\wedge p\ and\ s\vee p$
\end{center}

defining by successor, predecessor, successor-and-predecessor, and
successor-or-predecessor neighborhood operators, respectively. If $S_{i}$ ,$%
where$ $i:s$, $p,$ $s\wedge p$ $and$ $s\vee p,$ forms a subbase for some
topology on $U$, then the topology generated $S_{i}$ ,$where$ $i:s,$ $p,$ $%
s\wedge p$ $and$ $s\vee p$ denoted $T_{i},$ $where$ $i:s,$ $p,$ $s\wedge p$ $%
and$ $s\vee p,$ respectively.

A basic problem is: when does $\mathcal{S}_{i},where$ $i:s,$ $p,$ $s\wedge p$
$and$ $s\vee p,$ form a subbase for some topologies on $U$ ?

Our aim is to solve this problem completely. This problem was solved by the
authors in \cite{pei} using the family $\mathcal{S}_{s}$ forms a subbase for
some topology on $U$ by the following theorem.

\begin{theorem}
\cite{pei} If $R$ is a binary relation on $U$, then $S_{s}$ forms a subbase
for some topologies on $U$ if and only if $R$ is inverse serial.
\end{theorem}

\begin{remark}
It is clear that if $R$ is inverse serial, then

$U=\underset{x\in U}\bigcup R_{s}(x)$.

This is the condition (S0). Moreover, the family $\mathcal{S}_{s}$ is
covering of $U$.
\end{remark}

\begin{theorem}
\label{1}If $R$ is a binary relation on $U$, then $\mathcal{S}_{p}$ forms a
subbase for some topologies on $U$ if and only if $R$ is serial.
\end{theorem}

\begin{proof}
If $R$ is serial, then%
\begin{equation*}
U=\underset{x\in U}\bigcup R_{p}(x)\text{ }
\end{equation*}%
and the family $S_{p}$ provides the condition (S0).\qquad
\end{proof}

\begin{theorem}
If $R$ is a binary relation on $U$, then $\mathcal{S}_{s\wedge p}$ forms a
subbase for some topologies on $U$ if and only if $R$ is symmetric and
serial or inverse serial.
\end{theorem}

\begin{proof}
Similar to Theorem \ref{1}.
\end{proof}

\begin{theorem}
\label{8}If $R$ is a binary relation on $U$, then $\mathcal{S}_{s\vee p}$
forms a subbase for some topologies on $U$ if and only if $R$ is serial or
inverse serial.
\end{theorem}

\begin{proof}
If $R$ is serial or inverse serial, then%
\begin{equation*}
U=\underset{x\in U}\bigcup R_{s\vee p}(x)
\end{equation*}%
and the family $\mathcal{S}_{s\vee p}$ provides the condition (S0).
\end{proof}

Let $\mathcal{S}_{1}$ and $\mathcal{S}_{2}$ be covering of $U.$ A partition $%
\mathcal{S}_{1}$ is a finner than $\mathcal{S}_{2}$, or is coarser than $%
\mathcal{S}_{1}$, for each neighborhood operator in $\mathcal{S}_{1}$
produced by $x$, is subset the neighborhood operator in $\mathcal{S}_{2}$ by 
$x$. This relation is denoted as $\mathcal{S}_{1}$ $\preceq\mathcal{S}_{2}.$%
\begin{equation*}
\mathcal{S}_{1}\preceq\mathcal{S}_{2}\Longleftrightarrow\text{if every set
of }\mathcal{S}_{1}\text{ is contained in some sets of }\mathcal{S}_{2},%
\text{ for all }x\in U.
\end{equation*}

\begin{proposition}
\label{7}Let $U$ be the universe and $R$ is general binary relation. Then
follows as equivalent:

1) $R_{i}(x)\subseteq R_{j}(x),$

2)\ $\mathcal{S}_{i}\preceq \mathcal{S}_{j},$ for all $x\in U,$ where $%
i,j:s,p,s\wedge p$ and $s\vee p.$
\end{proposition}

\begin{proposition}
Let $U$ be the universe and $R$ is a serial relation. Then the following
conditions are provided:

1) $\mathcal{S}_{p}\preceq \mathcal{S}_{s\vee p},$

2) $T_{p}\preceq T_{s\vee p}.$
\end{proposition}

\begin{proof}
(1) is clear from Proposition \ref{7}.

(2) From Theorem \ref{1} and Theorem \ref{8}, $\mathcal{S}_{p}$ $\left( 
\mathcal{S}_{s\vee p}\right) $ forms a subbase for $T_{p}$ $\left( T_{s\vee
p}\right) $ topology on $U,$ respectively.
\end{proof}

\begin{proposition}
Let $U$ be the universe and $R$ is a inverse serial relation. Then the
following conditions are provided:

1) $\mathcal{S}_{s}\preceq \mathcal{S}_{s\vee p},$

2) $T_{s}\preceq T_{s\vee p}.$
\end{proposition}

\begin{proposition}
Let $U$ be the universe and $R$ is a symmetric relation. Then 
\begin{equation*}
R\text{ }is\text{ }a\text{ }serial\text{ }relation\Longleftrightarrow R\text{
}is\text{ }a\text{ }inverse\text{ }serial\text{ }relation.
\end{equation*}
\end{proposition}

\begin{proof}
Suppose that $R$ is a symmetric relation.

$R$ $is$ $a$ $serial$ $relation$ $\Longleftrightarrow \forall x\exists
y[\left( x,y\right) \in R]$

\ \ \ \ \ \ \ \ \ \ \ \ \ \ \ \ \ \ \ \ \ \ \ \ \ \ \ \ \ $%
\Longleftrightarrow \forall x\exists y[\left( y,x\right) \in R]$

\ \ \ \ \ \ \ \ \ \ \ \ \ \ \ \ \ \ \ \ \ \ \ \ \ \ \ \ \ $%
\Longleftrightarrow R$ $is$ $a$ $inverse$ $serial$ $relation.$
\end{proof}

\begin{proposition}
Let $U$ be the universe and $R$ is a symmetric and a serial (or inverse
serial) relation. Then the following conditions are provided: \ 

1) $\mathcal{S}_{s\wedge p}=\mathcal{S}_{s}=\mathcal{S}_{p}=\mathcal{S}%
_{s\vee p},$ \ 

2) $T_{s\wedge p}=T_{s}=T_{p}=T_{s\vee p}.$
\end{proposition}

\begin{corollary}
\label{2}Let $U$ be the universe and $R$ is a tolerance (symmetric and
reflexive) relation. Then, the following conditions are provided: \ 

1) $\mathcal{S}_{s\wedge p}=\mathcal{S}_{s}=\mathcal{S}_{p}=\mathcal{S}%
_{s\vee p},$ \ 

2) $T_{s\wedge p}=T_{s}=T_{p}=T_{s\vee p}.$
\end{corollary}

\begin{proposition}
Let $U$ be the universe and $R$ is a reflexive relation. Then, the following
conditions are provided: \ 

1) $\mathcal{S}_{s\wedge p}\preceq \mathcal{S}_{s},\mathcal{S}_{p}\preceq 
\mathcal{S}_{s\vee p},$ \ 

2) $T_{s\wedge p}\preceq T_{s},T_{p}\preceq T_{s\vee p}.$
\end{proposition}

\begin{corollary}
Let $U$ be the universe and $R$ is a preorder (reflexive and transitive)
relation. Then the following conditions are provided: \ 

1) $\mathcal{S}_{s\wedge p}\preceq \mathcal{S}_{s},\mathcal{S}_{p}\preceq 
\mathcal{S}_{s\vee p},$ \ 

2) $T_{s\wedge p}\preceq T_{s},T_{p}\preceq T_{s\vee p}.$
\end{corollary}

\begin{remark}
If $R$ is a preorder relation on $U$ , $\mathcal{S}_{s\wedge p}$ $\left( 
\mathcal{S}_{s},\mathcal{S}_{p},\mathcal{S}_{s\vee p}\right) $ form a base
for $T_{s\wedge p}$ $\left( T_{s},T_{p},T_{s\vee p}\right) $ topology on $U,$
respectively.
\end{remark}

\begin{corollary}
Let $U$ be the universe and $R$ is an equivalent relation. Then, the
following conditions are provided: \ 

1) $\mathcal{S}_{s\wedge p}=\mathcal{S}_{s}=\mathcal{S}_{p}=\mathcal{S}%
_{s\vee p},$ \ 

2) $T_{s\wedge p}=T_{s}=T_{p}=T_{s\vee p}$
\end{corollary}

\begin{remark}
In the case when $R$ is an equivalent relation on $U$, i.e, $\left(
U,R\right) $ is a Pawlak approximation space. Moreover, the set $\mathcal{S}%
_{s\wedge p}$ $\left( \mathcal{S}_{s},\mathcal{S}_{p},\mathcal{S}_{s\vee
p}\right) $ is a base for $T_{s\wedge p}$ $\left( T_{s},T_{p},T_{s\vee
p}\right) $ topology on $U.$ In these topologies, each neighborhood operator
is one equivalence class for all $x\in U$.
\end{remark}

\section{ Rough approximation operators induced by relations}

In this section, we investigate connection between lower and upper
approximation operators using successor, predecessor,
successor-and-predecessor, and successor-or-predecessor neighborhood
operators by various binary relations on a universe, respectively.

\begin{proposition}
\label{4}Let $U$ be the universe and $R$ is a binary relation. Then, for
lower and upper approximation operators of $X\subseteq U$, the following
conditions are provided: \ 

1) $\underline{apr}_{p\vee s}\left( X\right) \subseteq \underline{apr}%
_{s}\left( X\right) ,$ $\underline{apr}_{p}\left( X\right) \subseteq%
\underline{apr}_{p\wedge s}\left( X\right) ,$ \ 

2)$\overline{apr}_{p\wedge s}\left( X\right) \subseteq \overline{apr}%
_{s}\left( X\right) ,$ $\overline{apr}_{p}\left( X\right) \subseteq\overline{%
apr}_{p\vee s}\left( X\right) .$
\end{proposition}

\begin{proof}
$\left( 1\right) $ Let $\ x\in\underline{apr}_{p\vee s}\left( X\right) $ for
any $x\in U.$ Since $R_{p\vee s}\left( x\right) \subseteq X$ and $%
R_{s}\left( x\right) \subseteq R_{p\vee s}\left( x\right) $ then $%
R_{s}\left( x\right) \subseteq X$ and so $x\in\underline{apr}_{s}\left(
X\right) .$ Now, since $x\in\underline{apr}_{s}\left( X\right) $ and $%
R_{p\wedge s}\left( x\right) \subseteq R_{s}\left( x\right) \subseteq X$
then $x\in\underline{apr}_{p\wedge s}\left( X\right) .$ Therefore $%
\underline{apr}_{p\vee s}\left( X\right) \subseteq\underline{apr}_{s}\left(
X\right) $ $\subseteq\underline{apr}_{p\wedge s}\left( X\right) .$
Similarly, $\underline{apr}_{p\vee s}\left( X\right) \subseteq$ $\underline{%
apr}_{p}\left( X\right) \subseteq\underline{apr}_{p\wedge s}\left( X\right)
. $

$\left( 2\right) $ Let $\ x\in\overline{apr}_{p\wedge s}\left( X\right) $
for any $x\in U.$ Since $R_{p\wedge s}\left( x\right) \cap X\neq\emptyset$
and $R_{p\wedge s}\left( x\right) \subseteq R_{s}\left( x\right)
,R_{p}\left( x\right) \subseteq R_{p\vee s}\left( x\right) $ then $R_{p\vee
s}\left( x\right) \cap X\neq\emptyset$ \ an so $x\in\overline {apr}_{p\vee
s}\left( X\right) .$ Therefore, $\overline{apr}_{p\wedge s}\left( X\right)
\subseteq\overline{apr}_{s}\left( X\right) ,$ $\overline{apr}_{p}\left(
X\right) \subseteq\overline{apr}_{p\vee s}\left( X\right) .$
\end{proof}

\begin{example}
Let $U=\left\{ a,b,c,d\right\} $ and 
\begin{equation*}
R=\left\{ \left( a,a\right) ,\left( a,c\right) ,\left( b,c\right) ,\left(
c,a\right) ,\left( c,d\right) \right\}
\end{equation*}
be a binary relation on $U.$ Then,%
\begin{align*}
R_{s}\left( a\right) & =\left\{ a,c\right\} ,\text{ }R_{s}\left( b\right)
=\left\{ c\right\} ,\text{ }R_{s}\left( c\right) =\left\{ a,d\right\} ,\text{
}R_{s}\left( d\right) =\emptyset, \\
R_{p}\left( a\right) & =\left\{ a,c\right\} ,\text{ }R_{p}\left( b\right)
=\emptyset,\text{ }R_{p}\left( c\right) =\left\{ a,b\right\} ,\text{ }%
R_{p}\left( d\right) =\left\{ c\right\} , \\
R_{p\wedge s}\left( a\right) & =\left\{ a,c\right\} ,\text{ }R_{p\wedge
s}\left( b\right) =\emptyset,\text{ }R_{p\wedge s}\left( c\right) =\left\{
a\right\} ,\text{ }R_{p\wedge s}\left( d\right) =\emptyset, \\
R_{p\vee s}\left( a\right) & =\left\{ a,c\right\} ,\text{ }R_{p\vee s}\left(
b\right) =\left\{ c\right\} ,\text{ }R_{p\vee s}\left( c\right) =\left\{
a,b,d\right\} ,\text{ }R_{p\vee s}\left( d\right) =\left\{ c\right\} .
\end{align*}
Let $X=\left\{ a,c,d\right\} .$ Then,%
\begin{align*}
\underline{apr}_{s}\left( X\right) & =\left\{ a,b,c,d\right\} , \\
\underline{apr}_{p}\left( X\right) & =\left\{ a,b,d\right\} , \\
\underline{apr}_{p\wedge s}\left( X\right) & =\left\{ a,b,c,d\right\} , \\
\underline{apr}_{p\vee s}\left( X\right) & =\left\{ a,b,d\right\} , \\
\overline{apr}_{s}\left( X\right) & =\left\{ a,b,c\right\} , \\
\overline{apr}_{p}\left( X\right) & =\left\{ a,c,d\right\} , \\
\overline{apr}_{p\wedge s}\left( X\right) & =\left\{ a,c\right\} , \\
\overline{apr}_{p\vee s}\left( X\right) & =\left\{ a,b,c,d\right\} .
\end{align*}
Hence, note that $\underline{apr}_{s}\left(X\right)\supset \overline{apr}%
_{s}\left( X\right).$
\end{example}

In the original rough set theory, lower approximation of $X$ is a subset its
upper approximation. In order to provide this condition, we need some
properties to add binary relations.

\begin{proposition}
\cite{jarvinen}\label{3} Let $U$ be the universe and $R$ is a binary
relation. Then, for all $x\in U$%
\begin{equation*}
R\text{ is serial }\Rightarrow \underline{apr}_{s}\left( X\right) \subseteq 
\overline{apr}_{s}\left( X\right)
\end{equation*}
\end{proposition}

\begin{corollary}
Let $U$ be the universe and $R$ is a binary relation. Then, for all $x\in U$%
\begin{equation*}
R\text{ is serial }\Rightarrow\underline{apr}_{p\vee s}\left( X\right)
\subseteq\underline{apr}_{s}\left( X\right) \subseteq\overline{apr}%
_{s}\left( X\right) \subseteq\overline{apr}_{p\vee s}\left( X\right)
\end{equation*}
\end{corollary}

\begin{proof}
Proof is clear from Proposition \ref{4} and Proposition \ref{3}.
\end{proof}

\begin{proposition}
\label{5}Let $U$ be the universe and $R$ is a binary relation. Then, for all 
$x\in U$%
\begin{equation*}
R\text{ is invers serial }\Rightarrow \underline{apr}_{p}\left( X\right)
\subseteq \overline{apr}_{p}\left( X\right) .
\end{equation*}
\end{proposition}

\begin{proof}
Let $\ x\in\underline{apr}_{p}\left( X\right) .$ Then $R_{p}\left( x\right)
\subseteq X,$ which gives $R_{p}\left( x\right) \cap X=R_{p}\left( x\right)
\neq\emptyset,$ that is, $x\in\overline{apr}_{p}\left( X\right) .$
\end{proof}

\begin{corollary}
Let $U$ be the universe and $R$ is a binary relation. Then, for all $x\in U$%
\begin{equation*}
R\text{ is invers serial }\Rightarrow\underline{apr}_{p\vee s}\left(
X\right) \subseteq\underline{apr}_{p}\left( X\right) \subseteq \overline{apr}%
_{p}\left( X\right) \subseteq\overline{apr}_{p\vee s}\left( X\right) .
\end{equation*}
\end{corollary}

\begin{proof}
Proof is clear from proposition \ref{4} and proposition \ref{5} .
\end{proof}

\begin{proposition}
Let $U$ be the universe and $R$ is a binary relation. Then, for all $x\in U$%
\begin{equation*}
R\text{ is symmetric and serial (or invers serial)}\Rightarrow\underline {apr%
}_{p\wedge s}\left( X\right) \subseteq\overline{apr}_{p\wedge s}\left(
X\right) .
\end{equation*}
\end{proposition}

\begin{proof}
Let $\ x\in\underline{apr}_{p\wedge s}\left( X\right) .$ Then, from
proposition 14 $R_{p\wedge s}\left( x\right) =R_{p}\left( x\right) $ which
gives $\underline{apr}_{p\wedge s}\left( X\right) =\underline{apr}_{p}\left(
X\right) $. So, \ from proposition 24 $x\in\overline {apr}_{p\wedge s}\left(
X\right) .$
\end{proof}

\begin{proposition}
\label{6}Let $U$ be the universe and $R$ is a binary relation. Then, for all 
$x\in U$%
\begin{equation*}
R\text{ is reflexive}\Rightarrow \underline{apr}_{i}\left( X\right)
\subseteq X\subseteq \overline{apr}_{i}\left( X\right) ,\text{ where }i\text{%
: }s\text{, }p\text{, }p\wedge s\text{ and }p\vee s,\text{ respectively.}
\end{equation*}
\end{proposition}

\begin{proof}
Proof is clear from $(L3)$ and $(U3)$.
\end{proof}

\begin{corollary}
Let $U$ be the universe and $R$ is a reflexive or preorder binary relation.
Then, for all $x\in U$%
\begin{align*}
(1)\text{ }\underline{apr}_{p\vee s}\left( X\right) & \subseteq \underline{%
apr}_{p}\left( X\right) \subseteq\underline{apr}_{p\wedge s}\left( X\right)
\subseteq X\subseteq\overline{apr}_{p\wedge s}\left( X\right) \subseteq%
\overline{apr}_{p}\left( X\right) \subseteq\overline {apr}_{p\vee s}\left(
X\right) \\
(2)\text{ }\underline{apr}_{p\vee s}\left( X\right) & \subseteq \underline{%
apr}_{s}\left( X\right) \subseteq\underline{apr}_{p\wedge s}\left( X\right)
\subseteq X\subseteq\overline{apr}_{p\wedge s}\left( X\right) \subseteq%
\overline{apr}_{s}\left( X\right) \subseteq\overline {apr}_{p\vee s}\left(
X\right) \text{ .}
\end{align*}
\end{corollary}

\begin{proof}
Proof is clear from proposition \ref{3} and proposition \ref{6}.
\end{proof}

\begin{proposition}
Let $U$ be the universe and $R$ is a tolerance or equivalent binary
relation. Then, for all $x\in U$ \ 

$\underline{apr}_{s}\left( X\right) =\underline{apr}_{p}\left( X\right) =%
\underline{apr}_{p\wedge s}\left( X\right) =\underline{apr}_{p\vee s}\left(
X\right) \subseteq X\subseteq\overline{apr}_{s}\left( X\right)$ \ 

$=\overline{apr}_{p}\left( X\right) =\overline{apr}_{p\wedge s}\left(
X\right) =\overline{apr}_{p\vee s}\left( X\right) \text{ .}$
\end{proposition}

\begin{proof}
Proof is clear from corollary \ref{2} and proposition \ref{6}, respectively.
\end{proof}

\end{document}